\documentclass[12pt]{amsart}
\usepackage[cp1251]{inputenc}
\usepackage{amsmath}
\usepackage{amssymb}
\usepackage[english]{babel}
\usepackage{bbm}
\usepackage{stmaryrd}
\usepackage{amsthm,amsfonts}

\newtheorem{theorem}{Theorem}[section]
\newtheorem{claim}[theorem]{Proposition}

\newtheorem{lemma}[theorem]{Lemma}
\newtheorem{corollary}[theorem]{Corollary}

\begin{document}
\title{Derivations on ideals in commutative $AW^*$-algebras}

\author{V. I. Chilin }
\address{Department of Mathematics, National University of Uzbekistan,
Vuzgorodok, 100174, Tashkent, Uzbekistan}
\email{chilin@ucd.uz, vladimirchil@gmail.com}

\author{G. B. Levitina }
\address{Department of Mathematics, National University of Uzbekistan,
Vuzgorodok, 100174, Tashkent, Uzbekistan}
\email{bob\_galina@mail.ru}

\date{\today}
\begin{abstract}
Let $\mathcal{A}$ be a commutative $AW^*$-algebra, let $S(\mathcal{A})$ be the $*$-algebra of all measurable operators affiliated with $\mathcal{A}$, let $\mathcal{I}$ be an ideal in $\mathcal{A}$, let $s(\mathcal{I})$ be the support of the ideal $\mathcal{I}$ and let $\mathbb{Y}$ be a solid subspace in $S(\mathcal{A})$. The necessary and sufficient conditions of existence of non-zero band preserving derivations from $\mathcal{I}$ to $\mathbb{Y}$ are given. We show that, in case when $\mathbb{Y}\subset\mathcal{A}$, or $\mathbb{Y}$ is a quasi-normed solid space, any band preserving derivation from $\mathcal{I}$ into
$\mathbb{Y}$ is always trivial. At the same time, there exist non-zero band preserving derivations from $\mathcal{I}$ with values in  $S(\mathcal{A})$, if and only if the Boolean algebra of all projections from the $AW^*$-algebra $s(\mathcal{I})\mathcal{A}$ is not $\sigma$-distributive.
\end{abstract}
\maketitle
{\small Keywords: Boolean algebra, commutative $AW^*$-algebra, ideal, derivation.

\phantom{A}


\phantom{A}

}

\newcounter{gl}

\section{Introduction}
It is well known (\cite{Sak}, Lemma 4.1.3) that every derivation on a $C^*$-algebra is norm continuous and, when a $C^*$-algebra $\mathcal{A}$ is commutative, any derivation on $\mathcal{A}$ is trivial. In particular, for a commutative  von Neumann algebra and a commutative $AW^*$-algebra $\mathcal{A}$ any derivation $\delta\colon\mathcal{A}\to\mathcal{A}$ is identically zero. Development of the theory of algebras $S(\mathcal{A})$ of measurable operators affiliated with a von Neumann algebra or a $AW^*$-algebra $\mathcal{A}$ (\cite{Berberian}, \cite{Seg}) allowed to construct and  study the new significant examples of  $*$-algebras of unbounded operators. One of interesting problems in this theory is the problem of description of derivations acting in  $S(\mathcal{A})$. For a von Neumann algebra and an $AW^*$-algebra $\mathcal{A}$ it is known (\cite{Olesen}, (\cite{Sak}, 4.1.6)) that any derivation $\delta\colon\mathcal{A}\to\mathcal{A}$ is inner, i.e. it has a form $\delta(x)=[a,x]=ax-xa$ for some $a\in\mathcal{A}$ and all $x\in\mathcal{A}$. For the algebra $S(\mathcal{A})$ it is not the same. In case of a commutative von Neumann algebra $\mathcal{A}$ in  \cite{B-Ch-S-matzam},\cite{B-Ch-S} it was established that any derivation on $S(\mathcal{A})$ is inner, i.e. trivial, if and only if $\mathcal{A}$ is an atomic algebra. For a commutative  $AW^*$-algebra $\mathcal{A}$ there are non-zero derivations on $S(\mathcal{A})$, if and only if the Boolean algebra of all projections from $\mathcal{A}$ is not $\sigma$-distributive \cite{Kusraev}. In case of a von Neumann algebra $\mathcal{A}$ of type $I$, all derivations from $S(\mathcal{A})$ into $S(\mathcal{A})$ are described in \cite{AAK}. The next step in the study of properties of derivations in operator algebras has become the research of derivations acting on an ideal in a von Neumann algebra $\mathcal{A}$ with values in a Banach solid space in $S(\mathcal{A})$ \cite{B-S_d}. In particular, in \cite{B-S_d} it is proven that any derivation from a commutative von Neumann algebra $\mathcal{A}$ with values in a Banach solid space  $\mathbb{Y}\subset S(\mathcal{A})$ is always trivial.

In this paper we consider derivations acting on an ideal $\mathcal{I}$ in a commutative $AW^*$-algebra $\mathcal{A}$ (respectively, in the algebra $C(Q)$ of all continuous real-valued functions defined on the Stone space $Q$ of a complete Boolean algebra $\mathcal{B}$) with values in a solid space $\mathbb{Y}\subset S(\mathcal{A})$ (respectively, in $\mathbb{Y}\subset C_\infty(Q)$).

Denote by $s(x)$ the support of an element $x\in S(\mathcal{A})$ (respectively, $x\in C_\infty(Q)$) (see Section \ref{sec_prel}) and for any $E\subset S(\mathcal{A})$ (respectively, $E\subset C_\infty(Q)$) define the support of  $E$ by the equality
 $s(E)=\sup\{s(x):x\in E\}$. The support of the image $\delta(\mathcal{I})$ of a derivation $\delta\colon\mathcal{I}\to \mathbb{Y}$ is denoted by $s(\delta)$. A derivation $\delta$ is called band preserving (compare (\cite{G-K-K}, 1.1.1), (\cite{Kusraev}, 2.3)), if $s(\delta(x))\leqslant s(x)$ for all $x\in\mathcal{I}$. The inequality $s(\delta)\leqslant s(\mathcal{I})$ is a necessary and sufficient condition for a derivation $\delta\colon\mathcal{I}\to \mathbb{Y}$ to be band preserving (see Section \ref{sec_der}).

In this paper it is proven that there exits a non-zero band preserving derivation $\delta\colon\mathcal{I}\to \mathbb{Y}$ if and only if there exists a non-zero projection $e\in\mathcal{A}$ such that the following conditions hold:

1). $e\leqslant s(\mathcal{I})s(\mathbb{Y})$;

2). $e\mathbb{Y}=eS(\mathcal{A})$;

3). The Boolean algebra of projections from $e\mathcal{A}$ is not  $\sigma$-distributive.

Using this criterion it is established that in case, when $\mathbb{Y}\subset\mathcal{A}$ (respectively, $\mathbb{Y}\subset C(Q)$) or when $\mathbb{Y}$ is a quasi-normed solid space, any band preserving derivation  $\delta\colon\mathcal{I}\to \mathbb{Y}$ is always trivial.

At the same time, there exist non-zero band preserving derivations from $\mathcal{I}$ into  $S(\mathcal{A})$ (respectively,  in $C_\infty(Q)$), if and only if the Boolean algebra of all projections from $s(\mathcal{I})\mathcal{A}$ (respectively, $s(\mathcal{I})\mathcal{B}$) is not  $\sigma$-distributive  (compare \cite{Kusraev}). In particular, if $\mathcal{A}$ is a commutative von Neumann algebra, then any band preserving derivation  $\delta\colon\mathcal{I}\to S(\mathcal{A})$ is trivial, if and only if the algebra  $s(\mathcal{I})\mathcal{A}$ is atomic (compare \cite{B-Ch-S}).

\section{Preliminaries}\label{sec_prel}
Let $\mathcal{B}$ be an arbitrary complete Boolean algebra with zero $\mathbf{0}$ and unity $\mathbf{1}$. For an arbitrary non-zero $e\in\mathcal{B}$ set $\mathcal{B}_e=\{q\in\mathcal{B}:q\leqslant e\}$. With respect to the partial order induced from $\mathcal{B}$ the set $\mathcal{B}_e$ is a Boolean algebra with zero $\mathbf{0}$ and unity $e$.

A set $D\subset \mathcal{B}$ minorizes a subset $E\subset\mathcal{B}$ if for each non-zero $e\in E$ there exists $\mathbf{0}\neq q\in D$ such that $q\leqslant e$. We need the following important property of complete Boolean algebras.

\begin{theorem}[\cite{book-Kusraev}, 1.1.6]\label{principle}  Let $\mathcal{B}$ be a complete Boolean algebra, $\mathbf{0}\neq e\in\mathcal{B}$ and let $D$ minorizes the subset $\mathcal{B}_e$. Then there exists a disjoint subset $D_0\subset D$ such that $\sup D_0=e$.
\end{theorem}

A non-zero element $q$ from a Boolean algebra $\mathcal{B}$ is called an atom, if  $\mathcal{B}_e=\{\mathbf{0},q\}$. A Boolean algebra $\mathcal{B}$ is called atomic if, for every non-zero  $e\in \mathcal{B}$ there exists an atom $q\in \mathcal{B}$ such that $q\leqslant e$. Each complete atomic Boolean algebra is isomorphic to the Boolean algebra $2^\Delta$ of all subsets of the set $\Delta$ of all atoms in  $\mathcal{B}$ (\cite{Vl}, ch. \setcounter{gl}{3}\Roman{gl}, \S 2).

A Boolean algebra $\mathcal{B}$ is said to be continuous if $\mathcal{B}$ does not have atoms. If $\mathcal{B}$ is a complete Boolean algebra and  $\Delta$ is non-empty set of all atoms from $\mathcal{B}$, then for $e=\sup\Delta$ we have that $\mathcal{B}_e$ is an atomic Boolean algebra and $\mathcal{B}_{\mathbf{1}-e}$ is a continuous Boolean algebra.

Denote by $\mathbb{N}$ the set of all natural numbers and by $\mathbb{N^N}$ the set of all mappings from $\mathbb{N}$ into $\mathbb{N}$. A $\sigma$-complete Boolean algebra $\mathcal{B}$ is called \textit{$\sigma$-distributive}, if for any double sequence  $\{e_{n,m}\}_{n,m\in\mathbb{N}}$ in $\mathcal{B}$ the following condition holds
$$\bigvee\limits_{n\in\mathbb{N}}\bigwedge\limits_{m\in\mathbb{N}}e_{n,m}= \bigwedge\limits_{\varphi\in\mathbb{N^N}}\bigvee\limits_{n\in\mathbb{N}} e_{n,\varphi(n)}.$$
Each complete atomic Boolean algebra is $\sigma$-distributive (\cite{Kusraev},  3.1). In general, the converse does not hold. Moreover, there exist continuous $\sigma$-distributive complete Boolean algebras (\cite{book-Kusraev}, 5.1.7). At the same time, the following proposition holds.

\begin{claim}[\cite{G-K-K}, 5.3.3]\label{atomic-distrib}
If $\mathcal{B}$ is a continuous $\sigma$-complete Boolean algebra and there  exists a finite strictly positive countably additive measure $\mu$ on $\mathcal{B}$, then $\mathcal{B}$ is not $\sigma$-distributive.
\end{claim}

A complete Boolean algebra $\mathcal{B}$ is said to be multinormed, if the set of all
finite completely additive measures separates the points of $\mathcal{B}$ (\cite{book-Kusraev}, 1.2.9). If a Boolean algebra $\mathcal{B}$ is multinormed, then there exists a partition $\{e_i\}_{i\in I}$  of unity $\mathbf{1}$ such that the Boolean algebra $\mathcal{B}_{e_i}$ has a finite strictly positive countably additive measure for all $i\in I$ (\cite{book-Kusraev}, 1.2.10). Therefore, from Proposition \ref{atomic-distrib} we have the following

\begin{corollary}[\cite{G-K-K}, 5.3.3]\label{cor_atomic-distrib} A multinormed Boolean algebra $\mathcal{B}$ is  $\sigma$-distributive, if and only if $\mathcal{B}$ is an atomic Boolean algebra.
\end{corollary}

A complete Boolean algebra $\mathcal{B}$ is called \textit{properly non $\sigma$-distributive}, if for any non-zero $e\in\mathcal{B}$ the Boolean algebra $\mathcal{B}_e$ is not $\sigma$-distributive. By Corollary \ref{cor_atomic-distrib}, every multinorned continuous Boolean algebra  $\mathcal{B}$ is properly non $\sigma$-distributive.

Let $Q$ be the Stone space of a complete Boolean algebra $\mathcal{B}$. Denote by $C_\infty(Q)$ the set of
all continuous functions $x: Q \rightarrow [-\infty, +\infty]$  assuming the values $\pm\infty$ possibly on a nowhere-dense set. The space $C_\infty(Q)$ with naturally defined algebraic operations and partial order is an algebra over the field $\mathbb{R}$ of real numbers and is a universally complete vector lattice. The identically one function
$\mathbf{1}_Q$ is the unity of the algebra $C_\infty(Q)$ and an order-unity of the vector lattice $C_\infty(Q)$ (\cite{book-Kusraev}, 1.4.2).

An element $x\in C_\infty(Q)$ is an idempotent, i.e. $x^2=x$, if and only if  $$x(t)=\chi_{Q(e)}(t)=\left\{\begin{aligned}1,t\in Q(e)\\0,t\notin Q(e)\end{aligned}\right.=:e(t)$$
for some clopen set $Q(e)\subset Q$ corresponding to the element $e\in \mathcal{B}$, in addition, $e\leqslant q \Leftrightarrow e(t)\leqslant q(t)$ for all $t\in Q$, where $e,q\in \mathcal{B}$. Thus, the Boolean algebra $\mathcal{B}$ is identified with the Boolean algebra of all idempotents from $C_\infty(Q)$. In this identification the unity $\mathbf{1}$ of $\mathcal{B}$ coincides with the function $\mathbf{1}_Q$, and zero $\mathbf{0}$ of $\mathcal{B}$ coincides with identically zero function. Further we suppose that $\mathcal{B}\subset C_\infty(Q)$, and the algebra $C_\infty(Q)$ we denote by $L^0(\mathcal{B})$.

As in an arbitrary vector lattice, for every $x\in L^0(\mathcal{B})$ denote by $x_+:=x\vee 0$ (respectively, $x_-:=-(x\wedge 0)$) the positive (negative) part of $x$ and by $|x|:=x_++x_-$ denote the modulus of $x$. The set of all positive elements of $L^0(\mathcal{B})$ is denoted by $L^0_+(\mathcal{B})$.

For every $x\in L^0(\mathcal{B})$ define the support of $x$ by the equality \\$s(x)=\mathbf{1}-\sup\{e\in\mathcal{B}: ex=0\}$. It is clear that $s(x)\in\mathcal{B}$ and $s(q)=q$ for all  $q\in\mathcal{B}$. Note, that an idempotent $q\in\mathcal{B}$ is the support of $x\in L^0(\mathcal{B})$, if and only if $qx=x$ and from $e\in\mathcal{B}, ex=x$, it follows that $e\geqslant q$.

It is easy to see that supports of elements have the following properties:

\begin{claim}\label{pr5} If $x,y\in L^0(\mathcal{B}), 0\neq\lambda\in \mathbb{R}$, then

(i). $s(\lambda x)= s(x);$

(ii). $s(xy)=s(x)s(y)$;

(iii). $s(|x|)=s(x)$;

(iv). If $xy=0$, then $s(x+y)=s(x)+s(y)$, in particular, $s(x)=s(x_+)+s(x_-)$.
\end{claim}

For an arbitrary non-empty subset $E\subset L^0(\mathcal{B})$ define the support $s(E)$ of $E$ by setting $s(E)=\sup\{s(x)\colon x\in E\}$.

A non-zero linear subspace $X$ of $L^0(\mathcal{B})$ is called a solid space, if $x\in X, y\in L^0(\mathcal{B})$ and $|y|\leqslant|x|$ implies that $y\in X$. If $X$ is a solid space in $L^0(\mathcal{B})$ and $s(X)=\mathbf{1}$, then $X$ is said to be a order-dense solid space in $L^0(\mathcal{B})$.

Denote by $C(Q)$ the algebra of all continuous functions on $Q$ with values in $\mathbb{R}$. It is clear that $C(Q)$ is a subalgebra and a order-dense solid space in $L^0(\mathcal{B})$, in addition, $C(Q)$ is a Banach algebra with the norm $\|x\|_\infty=\sup\limits_{t\in Q}|x(t)|,\\x\in C(Q)$. As usual, a subalgebra $\mathcal{A}$ of $C(Q)$ is called an ideal, if $xa\in \mathcal{A}$ for all $x\in C(Q), a\in\mathcal{A}$.

\begin{claim}\label{pr6} (i). A linear subspace $X$ in $L^0(\mathcal{B})$ is solid, if and only if $C(Q) X=X$;

(ii). If $X$ is a solid space and $X\subset C(Q)$, then $X$ is an ideal in $C(Q)$, in particular, $X$ is a subalgebra of $C_\infty(Q)$. Conversely, if $X$ is an ideal in $C(Q)$, then $X$ is a solid space in $L^0(\mathcal{B})$.
\end{claim}
\begin{proof} (i) It is clear that $X=\mathbf{1}X\subset C(Q) X$. Let $X$ be a solid space in $L^0(\mathcal{B})$ and let $y\in C(Q), x\in X$. Select $c>0$ such that $|y|\leqslant c \mathbf{1}$. Then $|yx|\leqslant c|x|$, that implies $yx\in X$. Consequently, $C(Q) X\subset X$, and therefore $C(Q) X=X$.

Conversely, if $C(Q) X=X$, $x\in X, y\in L^0(\mathcal{B})$ and $|y|\leqslant |x|$, then $|y|=a|x|$, where $a\in C(Q)$ and $0\leqslant a\leqslant \mathbf{1}$. Hence, $|y|\in X$, and therefore $y=(s(y_+)-s(y_-))|y|\in X$.

(ii) If $X$ is a solid space and $X\subset C(Q)$, then $C(Q) X=X$ (see item (i)), and therefore $X$ is an ideal in $C(Q)$. Conversely, if $X$ is an ideal in $C(Q)$, then $C(Q) X=X$ and, by item (i), $X$ is a solid space in $L^0(\mathcal{B})$.
\end{proof}

\begin{claim}\label{fund}Let $X$ be a solid space in  $L^0(\mathcal{B})$. Then

(i). For any non-zero $e\in s(X)\mathcal{B}$ there exists $0\neq p\in X\cap\mathcal{B}$ such that $p\leqslant e$.

(ii). There exists a partition $\{e_i\}_{i\in I}$ of the support $s(X)$ contained in $X$.
\end{claim}
\begin{proof}
(i). Since $s(X)=\sup\{s(x):x\in X\}$, for every given $0\neq e\in s(X)\mathcal{B}$ there exists a non-zero $x_0\in X$ such that $s(|x_0|)e\neq 0$. For every $\lambda>0$ consider the spectral idempotent $e_\lambda(|x_0|)=\{t\in Q:|x_0(t)|\leqslant\lambda\}$ of $|x_0|$. Since  $\lambda\bigl(\mathbf{1}-e_\lambda(|x_0|)\bigr)\leqslant \bigl(\mathbf{1}-e_\lambda(|x_0|)\bigr)|x_0|$ and $X$ is a solid space in $L^0(\mathcal{B})$, it follows that $\bigl(\mathbf{1}-e_\lambda(|x_0|)\bigr)\in X$. Using the convergence $\bigl(\mathbf{1}-e_\lambda(|x_0|)\bigr)\uparrow s(|x_0|)$ for $\lambda\rightarrow 0$, we have$\bigl(\mathbf{1}-e_\lambda(|x_0|)\bigr) e\uparrow s(|x_0|)e\neq 0$. Hence, there exists $\lambda_0>0$, such that $p=\bigl(\mathbf{1}-e_{\lambda_0}(|x_0|)\bigr)e\neq 0$,  in addition $p\in X$ and $p\leqslant e$.

Item (ii) follows from Theorem \ref{principle} and item (i).
\end{proof}

Let $X$ be a linear space over the field $\mathbb{K}$, where $\mathbb{K}$ is either the field $\mathbb{C}$ of complex numbers, or the field  $\mathbb{R}$ of real numbers. A mapping $\|\cdot\|: X\to \mathbb{R}$ is called a quasi-norm, if there exists $C\geqslant 1$ such that for all
$x,y\in X,\alpha \in\mathbb{K}$ the following properties hold

1) $\|x\|\geqslant 0, \|x\|=0 \Leftrightarrow x=0$;

2) $\|\alpha x\|=|\alpha|\|x\|$;

3) $\|x+y\|\leqslant C(\|x\|+\|y\|)$.


A quasi-norm $\|\cdot\|_X$ on a solid space $X$ is said to be monotone, if  $x,y\in X, |y|\leqslant|x|$ implies that $\|y\|_X\leqslant\|x\|_X$. A solid space $X$ in $L^0(\mathcal{B})$ is called a quasi-normed solid space, if $X$ is equipped with a monotone quasi-norm.

For every non-zero $e\in\mathcal{B}$ consider the subalgebra \\$eL^0(\mathcal{B})=\{ex:x\in L^0(\mathcal{B})\}$. If $Q$ is the Stone space of a complete Boolean algebra $\mathcal{B}$, and
$Q(e)$ is a clopen set corresponding to the idempotent $e\in\mathcal{B}$, then $Q_e=Q(e)\cap Q$ is the Stone space of the Boolean algebra $\mathcal{B}_e$. In addition, the mapping $\Phi\colon eL^0(\mathcal{B})\to L^0(\mathcal{B}_e)$ defined by the equality $\Phi(x)=x|_{Q_e}$, where $x|_{Q_e}$ is the restriction of a function $x$ on the compact $Q_e$, is an algebraic and lattice isomorphism of  $eL^0(\mathcal{B})$ onto $L^0(\mathcal{B}_e)$, such that $\Phi(p)=p$ for all $p\in e\mathcal{B}$. Consequently, for every solid space $X$ in $L^0(\mathcal{B})$ and every $0\neq e\in\mathcal{B}$ the image $Y=\Phi(eX)$ is a solid space in $L^0(\mathcal{B}_e)$, in addition, if $\|\cdot\|_X$ is a monotone quasi-norm on $X$, then $\|y\|_Y=\|\Phi^{-1}(y)\|_X,y\in\Phi(eX),$ is a monotone quasi-norm on  $Y$. Hence the following proposition holds.

\begin{claim}\label{eY} If $X$ is a solid space (respectively, $(X, \|\cdot\|_X)$ is a quasi-normed solid space) in $L^0(\mathcal{B}),0\neq e\in\mathcal{B}$, then $\Phi(eX)$ (respectively, \\$(\Phi(eX),\|\cdot\|_{\Phi(eX)})$ ) is a solid space (respectively, a quasi-normed solid space) in $L^0(\mathcal{B}_e)$, where $\|\Phi(ex)\|_{\Phi(eX)}=\|ex\|_X$ for all $x\in X$. In addition, $\Phi(eX)$ is a order-dense solid space in $L^0(\mathcal{B}_e)$, if $e=s(X)$.
\end{claim}

Now, consider the complexification $L^0_\mathbb{C}(\mathcal{B})=L^0(\mathcal{B})\varoplus iL^0(\mathcal{B})$ (with $i$ standing for the imaginary unity) of a real vector lattice $L^0(\mathcal{B})$ (see \cite{book-Kusraev}, 1.3.11). As usual, for an element $z=x+iy,x,y\in L^0(\mathcal{B})$ the adjoint element of $z$ is defined by the equality $\overline{z}=x-iy$, in addition, $\mathrm{Re}z=\frac{1}{2}(z+\overline{z})=x$ is called the real part of  $z$ and $\mathrm{Im}z=\frac{1}{2i}(z-\overline{z})=y$ is called the imaginary part of $z$. The modulus $|z|$ of every element $z\in L^0_\mathbb{C}(\mathcal{B})$ is defined by the equality $|z|:=\sup\{\mathrm{Re}(e^{i\theta}z):0\leqslant\theta<2\pi\}.$ An element $z$ of $L^0_\mathbb{C}(\mathcal{B})$ may be interpreted as continuous function $z\colon Q\to \mathbb{\overline{C}}=\mathbb{C}\cup \{\infty\}$, assuming the values $\infty$ possibly on a nowhere-dense set, where \  $\mathbb{\overline{C}}$ is the one-point compactification of $\mathbb{C}$. In addition, the algebraic operations in $L^0_\mathbb{C}(\mathcal{B})$ coincide with pointwise algebraic operations for the functions from $L^0_\mathbb{C}(\mathcal{B})$, defined up to non-where dense sets. In particular, $L^0_\mathbb{C}(\mathcal{B})$ is a commutative $*$-algebra and the modulus  $|z|$ of an element $z\in L^0_\mathbb{C}(\mathcal{B})$ is defined by the equality $|z|(t)=(\overline{z}(t)z(t))^{\frac{1}{2}}=\left((\mathrm{Re}z(t))^2+(\mathrm{Im}z(t))^2\right)^{\frac{1}{2}}$ for all $t$ from some open everywhere dense set from $Q$. The selfadjoint part $(L^0_\mathbb{C}(\mathcal{B}))_h=\{z\in L^0_\mathbb{C}(\mathcal{B}):z=\overline{z}\}$ of a complex vector lattice $L^0_\mathbb{C}(\mathcal{B})$ coincides with $L^0(\mathcal{B})$. In particular, $|z|\in L^0_+(\mathcal{B})$ for all $z\in L^0_\mathbb{C}(\mathcal{B})$. The algebra $\mathbb{C}(Q)$ of all continuous complex functions $z\colon Q\to \mathbb{C}$ coincides with the complexification $C(Q)\varoplus iC(Q)$ and $\mathbb{C}(Q)$ is a commutative $C^*$-algebra with the norm $\|z\|_\infty=\sup\limits_{t\in Q} |z(t)|$.

For any $z\in L^0_\mathbb{C}(\mathcal{B})$ the support of $z$ is defined by the equality $s(z):=s(|z|)$, and the support $s(E)$ of a subset $E\subset L^0_\mathbb{C}(\mathcal{B})$, by the equality \\$s(E)=\sup\{s(z):z\in E\}$, respectively.

A solid space $\mathbb{X}$ and a quasi-normed solid space $(\mathbb{X},\|\cdot\|_\mathbb{X})$ in $L^0_\mathbb{C}(\mathcal{B})$ are defined as in $L^0(\mathcal{B})$. It is clear that for a solid space $\mathbb{X}$ in $L^0_\mathbb{C}(\mathcal{B})$ the set $X:=\mathbb{X}_h=\mathbb{X}\cap L^0(\mathcal{B})$ is a solid space in  $L^0(\mathcal{B})$, in addition, $\mathbb{X}=X\varoplus iX$. Moreover, an element $z\in L^0_\mathbb{C}(\mathcal{B})$ contained in $\mathbb{X}$ if and only if $|z|\in \mathbb{X}_h\cap L^0_+(\mathcal{B})$, in addition, $s(\mathbb{X})=s(\mathbb{X}_h)$. If $\|\cdot\|_\mathbb{X}$ is a monotone quasi-norm on  $\mathbb{X}$, then $(X,\|\cdot\|_\mathbb{X})$ is a quasi-normed solid space in  $L^0(\mathcal{B})$.

Conversely, if $X$ is a solid space in $L^0(\mathcal{B})$ and $\|\cdot\|_X$ is a monotone quasi-norm on  $X$, then $\mathbb{X}=X\varoplus iX$ is a solid space in $L^0_\mathbb{C}(\mathcal{B}),\mathbb{X}_h=X$ and the function $\|z\|_\mathbb{X}=\| |z| \|_X,z\in\mathbb{X}$ is a monotone quasi-norm on $\mathbb{X}$.

\section{Derivations on solid spaces in  $L^0(\mathcal{B})$}\label{sec_der}

A linear mapping  $\delta$ from $L^0(\mathcal{B})$ (respectively, $L^0_\mathbb{C}(\mathcal{B})$) into $L^0(\mathcal{B})$ (respectively, $L^0_\mathbb{C}(\mathcal{B})$) is called a \textit{derivation} if
\begin{gather}\label{def-der}
\delta(xy)=\delta(x)y+x\delta(y)
\end{gather}
 for all $x,y\in L^0(\mathcal{B})$ (respectively, $x,y\in L^0_\mathbb{C}(\mathcal{B})$).

The following theorem gives a necessary and sufficient condition for existence of non-zero derivations.

\begin{theorem}[\cite{Kusraev}, Corollary 3.5]\label{Kus} For a complete Boolean algebra $\mathcal{B}$ the following conditions are equivalent:

(i). $\mathcal{B}$ is a $\sigma$-distributive Boolean algebra;

(ii). There are no non-zero derivations from $L^0_\mathbb{C}(\mathcal{B})$ into $L^0_\mathbb{C}(\mathcal{B})$.
\end{theorem}

In case, when the Boolean algebra $\mathcal{B}$ is multinormed, $\mathcal{B}$ is $\sigma$-distributive if and only if $\mathcal{B}$ is an atomic Boolean algebra $\mathcal{B}$ (see Corollary \ref{cor_atomic-distrib}). Therefore for a multinormed Boolean algebra $\mathcal{B}$, there exist nonzero derivations on  $L^0_\mathbb{C}(\mathcal{B})$, if and only if the Boolean algebra $\mathcal{B}$ is not atomic (this fact was also established independently of  \cite{Kusraev} in  (\cite{B-Ch-S}, Theorem 3.4)).

By Theorem \ref{Kus}, in case when $\mathcal{B}$ is not a $\sigma$-distributive Boolean algebra there exists non-zero derivations from $L^0_\mathbb{C}(\mathcal{B})$ into $L^0_\mathbb{C}(\mathcal{B})$. 

A derivation $\delta\colon L^0_\mathbb{C}(\mathcal{B})\to L^0_\mathbb{C}(\mathcal{B})$ is called a $*$-derivation, if $\delta(\overline{z})=\overline{\delta(z)}$ for all $z\in L^0_\mathbb{C}(\mathcal{B})$. In this case, $\delta(L^0(\mathcal{B}))\subset L^0(\mathcal{B})$ and the restriction of $\delta$ on $L^0(\mathcal{B})$ is a (real) derivation from $L^0(\mathcal{B})$ into $L^0(\mathcal{B})$.

For every derivation $\delta\colon L^0_\mathbb{C}(\mathcal{B})\to L^0_\mathbb{C}(\mathcal{B})$ consider the mappings
$\delta_\mathbbm{Re}$ and $\delta_\mathbbm{Im}$ from $L^0_\mathbb{C}(\mathcal{B})$ into $L^0_\mathbb{C}(\mathcal{B})$, defined by the equalities:
\begin{gather}\label{self_adjoint_der}
\delta_\mathbbm{Re}(x)=\frac{\delta(x)+\overline{\delta(x)}}{2}, \delta_\mathbbm{Im}(x)=\frac{\delta(x)-\overline{\delta(x)}}{2i},x\in L^0_\mathbb{C}(\mathcal{B}).
\end{gather}
It is clear that $\delta_\mathbbm{Re},\delta_\mathbbm{Im}$ are $*$-derivations from $L^0_\mathbb{C}(\mathcal{B})$ into $L^0_\mathbb{C}(\mathcal{B})$, in addition $\delta=\delta_\mathbbm{Re}+i\delta_\mathbbm{Im}$.

Theorem \ref{Kus} implies the following
\begin{corollary}\label{cor-Kus} For a complete Boolean algebra $\mathcal{B}$ the following conditions are equivalent:

(i). The Boolean algebra $\mathcal{B}$ is $\sigma$-distributive;

(ii). There are no non-zero derivations from $L^0(\mathcal{B})$ into $L^0(\mathcal{B})$.
\end{corollary}
\begin{proof}
$(i)\Rightarrow(ii)$. If $\delta$ is an arbitrary (real) derivation from  $L^0(\mathcal{B})$ into $L^0(\mathcal{B})$, then $\hat{\delta}(x+iy)=\delta(x)+i\delta(y)$ is a (complex) derivation from $L^0_\mathbb{C}(\mathcal{B})$ into $L^0_\mathbb{C}(\mathcal{B})$. By Theorem \ref{Kus}, we have $\hat{\delta}=0$, and therefore $\delta=0$.

$(ii)\Rightarrow(i)$. Since any derivation $\delta\colon L^0_\mathbb{C}(\mathcal{B})\to L^0_\mathbb{C}(\mathcal{B})$ has a form  $\delta=\delta_\mathbbm{Re}+i\delta_\mathbbm{Im}$, where the restriction of  $\delta_\mathbbm{Re}$ and $\delta_\mathbbm{Im}$ on $L^0(\mathcal{B})$ are (real) derivations from $L^0(\mathcal{B})$ into $L^0(\mathcal{B})$, it follows that there are no non-zero derivations on $L^0_\mathbb{C}(\mathcal{B})$. By Theorem \ref{Kus}, the Boolean algebra $\mathcal{B}$ is not $\sigma$-distributive.
\end{proof}

Note, that the assertion of Corollary \ref{cor-Kus} has been mentioned in 5.3.4 of \cite{G-K-K}. Since in \cite{G-K-K} it is given without the proof, for complete expression we give the proof.

By Corollary \ref{cor_atomic-distrib}, for a multinormed Boolean algebra $\mathcal{B}$, there are no non-zero derivations on $L^0(\mathcal{B})$ if and only if the Boolean algebra is $\mathcal{B}$ atomic.

Let $X,Y$ be solid spaces in $L^0(\mathcal{B})$, $X\subset C(Q)$. By Proposition \ref{pr6}, \\$C(Q)Y=Y$, i.e. $xy\in Y$ for all $x\in X,y\in Y$. This circumstance provides the correctness of the following definition.

A linear mapping $\delta\colon X\to Y$ is called \textit{derivation}, if equality (\ref{def-der}) holds for all $x,y\in X$. A (complex) derivation for solid spaces $\mathbb{X}$ and $\mathbb{Y}$ in $L^0_\mathbb{C}(\mathcal{B})$, where $\mathbb{X}\subset \mathbb{C}(Q)$, is defined in the same manner.
For every derivation $\delta\colon X\to Y$ (respectively, $\delta\colon \mathbb{X}\to\mathbb{Y}$) the idempotent $s(\delta)=\sup\{s(\delta(x)):x\in X\}$ (respectively, $s(\delta)=\sup\{s(\delta(z)):z\in\mathbb{X}\}$) is called the support of the derivation $\delta$. It is clear that $s(\delta)\leqslant s(Y)$ (respectively, $s(\delta)\leqslant s(\mathbb{Y})$).

\begin{claim}[compare \cite{B-Ch-S}, \S 2, Proposition 2.3]\label{der} If $\delta\colon X\rightarrow Y$ is a derivation from $X$ into $Y$, then for all $e\in X\cap\mathcal{B}, x\in X$ the equalities $\delta(e)=0$ and $\delta(ex)=e\delta(x)$ hold.
\end{claim}
\begin{proof} If $e\in X\cap\mathcal{B}$, then  $\delta(e)=\delta(e^2)=\delta(e)e+e\delta(e)=2e\delta(e)$, and therefore  $e\delta(e)=2e\delta(e)$, i.e. $e\delta(e)=0$, that implies the equality $\delta(e)=0$. Further, for  $x\in X$ we have
$\delta(ex)=\delta(e)x+e\delta(x)=e\delta(x)$.
\end{proof}

Let $\mathcal{B}$ be a complete Boolean algebra, let $Q$ be the Stone space of $\mathcal{B}$. We need the following useful property of derivations on $L^0(\mathcal{B})$.

\begin{theorem}\label{sequence}If $\delta\colon C(Q)\to L^0(\mathcal{B})$ is a non-zero derivation, then there exist a sequence $\{a_n\}_{n=1}^\infty$ in $C(Q)$ and a non-zero idempotent $q\in\mathcal{B}$ such that $|a_n|\leqslant\mathbf{1}$ and $|\delta(a_n)|\geqslant nq $ for all $ n\in \mathbb{N}$.
\end{theorem}

In the proof of Theorem \ref{sequence} we use the notion of a cyclic set in $L^0(\mathcal{B})$ given below.

Let $\{e_i\}_{i\in I}$ be a partition of unity in  $\mathcal{B},x_i\in L^0(\mathcal{B}),i\in I$, where $I$ is an arbitrary index set. Select a unique $x\in L^0(\mathcal{B})$ such that $e_ix=e_ix_i$ for all $i\in I$ (the uniqueness of $x$ follows from the equality $\sup\limits_{i\in I} e_i=\mathbf{1}$ (\cite{book-Kusraev}, 7.3.1)). The element $x$ is called a mixing of the family $\{x_i\}_{i\in I}$ by the partition of unity $\{e_i\}_{i\in I}$ and is denoted by $\underset{i \in I}{\mathrm{mix}}(e_ix_i)$. If $x_i\geqslant 0$ for all ${i\in I}$, then $\underset{i \in I}{\mathrm{mix}}(e_ix_i)=\sup\limits_{i\in I} e_ix_i$.

The set of all mixing of families of elements from $E\subset L^0(\mathcal{B})$ is called a cyclic envelope of a subset $E$ in $L^0(\mathcal{B})$ and is denoted by  $\mathrm{mix}(E)$. Clearly, $E\subset \mathrm{mix}(E)$. If $\mathrm{mix}(E)=E$, then $E$ is said to be a cyclic subset in $L^0(\mathcal{B})$ (\cite{book-Kusraev-2}, 0.3.5). It is clear that $\mathrm{mix(mix}(E))=\mathrm{mix}(E)$. Since for every $x\in L^0(\mathcal{B})$ there exists a partition $\{e_n\}_{n=1}^\infty$ such that $e_nx\in C(Q)$ for all $n$, it follows that $x=\underset{n}{\mathrm{mix}}(e_nx)\in \mathrm{mix}(C(Q))$. Hence, $\mathrm{mix}(C(Q))=L^0(\mathcal{B})$.

\begin{lemma}\label{unbounded} If $E\subset L^0_+(\mathcal{B}),E=\mathrm{mix}(E)$ and $E$ is an order unbounded set in  $L^0(\mathcal{B})$, there exist $0\neq q\in\mathcal{B},\{x_n\}_{n=1}^\infty\subset E$ such that $qx_n\geqslant nq$ for all $n\in\mathbb{N}$.
\end{lemma}
\begin{proof}[Proof of the lemma] Let us show that there exists a non-zero idempotent $q\in\mathcal{B}$ such that the set $pE$ is an order unbounded in $L^0(\mathcal{B})$ for all $0\neq p\in \mathcal{B}_q$. If this is not the case, for every  $0\neq q\in\mathcal{B}$ there exist $0\neq p_q\in\mathcal{B}_q, a_p\in p_qL^0_+(\mathcal{B})$ such that $0\leqslant p_q x\leqslant a_p$ for all $x\in E$. By Theorem \ref{principle}, there exist a partition of unity $\{e_i\}_{i\in I}$ and elements $a_i\in e_iL^0_+(\mathcal{B})$, such that $0\leqslant e_ix\leqslant a_i$ for all $x\in E$ and $i\in I$. Set $a=\underset{i \in I}{\mathrm{mix}}(e_ia_i)$. Then $a\in L^0_+(\mathcal{B})$ and for $x\in E$ the relationships hold $e_ix\leqslant e_i a_i=e_i a$ for all $i\in I$, that implies $0\leqslant x\leqslant a$. It means that $E$ is order bounded in  $L^0(\mathcal{B})$, which is a contradiction. Consequently, there exists a non-zero idempotent  $q\in\mathcal{B}$ such that $pE$ is order unbounded in  $L^0(\mathcal{B})$ for all $0\neq p\in\mathcal{B}_q$.

Fix $n\in\mathbb{N}$ and for every $0\neq p\in\mathcal{B}_q$ select an element $x_{n,p}\in pE$ which is not dominated by the element $np$. It means that there exists  $0\neq r_p\in\mathcal{B}_p$ such that $r_px_{n,p}\geqslant nr_p$. Using Theorem \ref{principle} again, select a partition  $\{z_j\}_{j\in J}$ of the element $q$ and elements $x_{n,j}\in z_jE$, such that $nz_j\leqslant z_jx_{n,j}$. Since $E$ is a cyclic subset, then $x_n=\underset{j \in J}{\mathrm{mix}}(z_jx_{n,j})\in qE$, in addition,
$z_jx_n=z_jx_{n,j}\geqslant nz_j$ for all $j\in J$. Hence $qx_n=x_n\geqslant nq$ for all $n\in\mathbb{N}$.
\end{proof}

Let us proceed to the proof of Theorem \ref{sequence}.

Assume that the set $E=\{|\delta(a)|:a\in C(Q),|a|\leqslant\mathbf{1}\}$ is order bounded in  $L^0(\mathcal{B})$, i.e. there exists $x\in L^0_+(\mathcal{B})$ such that $|\delta(a)|\leqslant x$ for all $a\in C(Q)$ with $|a|\leqslant\mathbf{1}$. Let us show that in this case $\delta=0$.

Let $\{y_n\}\in C(Q)$ and $t_n=\|y_n-y\|_\infty\rightarrow 0$ for some $y\in C(Q)$. Then
\begin{gather}\label{(2)}
|\delta(y_n)-\delta(y)|=|\delta(y_n-y)|=t_n\left|\delta\left(\frac{y_n-y}{t_n}\right)\right|\leqslant t_n x
\end{gather}
for all $n\in\mathbb{N}$ with $t_n\neq 0$.

Since $\delta(e)=0$ for all $e\in\mathcal{B}$ (see Proposition \ref{der}), it follows that $\delta(x)=0$ for all step elements $x=\sum\limits_{i=1}^n\lambda_ie_i$, where $\lambda_i\in\mathbb{R},e_i\in\mathcal{B},i=\overline{1,n},n\in\mathbb{N}$. For every $b\in C(Q)$ there exists a sequence of step elements  $\{x_n\}_{n=1}^\infty$ such that $\|b-x_n\|_\infty\rightarrow 0$. Due to (\ref{(2)}), we have  $|\delta(b)|=|\delta(b)-\delta(x_n)|\leqslant\|b-x_n\|_\infty \,x$, that implies $\delta(b)=0$, which is contradiction to the assertion of Theorem \ref{sequence}.

Thus, the set  $E$ is order unbounded in $L^0(\mathcal{B})$. Let us show that $E=\mathrm{mix}(E)$. Let $\{e_i\}_{i\in I}$ be a partition of unity and let $\{x_i\}_{i\in I}$ be a family of elements from $E$. Since  $x_i=|\delta(a_i)|,a_i\in C(Q),|a_i|\leqslant\mathbf{1}$, then $e_ix_i=|e_i\delta(a_i)|=|\delta(e_ia_i)|,i\in I$. Setting $a=\underset{i \in I}{\mathrm{mix}}(e_ia_i)$ we have that $a\in C(Q),|a|\leqslant\mathbf{1}$ and
$$e_i|\delta(a)|=|e_i\delta(a)|=|\delta(e_ia)|=|\delta(e_ia_i)|=e_ix_i,i\in I, $$
i.e. $\underset{i \in I}{\mathrm{mix}}(e_ix_i)=|\delta(a)|\in E$. Consequently, $E=\mathrm{mix}(E)$. Therefore, by Lemma \ref{unbounded}, there exist $0\neq q\in\mathcal{B},\{a_n\}\in C(Q)$ with $|a_n|\leqslant\mathbf{1}$, such that \\$|\delta(a_n)|\geqslant  q|\delta(a_n)|\geqslant nq$ for all $n\in\mathbb{N}$.\qquad\qquad\qquad\qquad\qquad\qquad\qquad\qquad\qquad$\Box$

Theorem \ref{sequence} implies the following significant property of the support $s(\delta)$ of a derivation $\delta$ from $C(Q)$ into $L^0(\mathcal{B})$.
\begin{claim}
\label{ad_claim_foralln}
Let $\mathcal{B}$ be a complete Boolean algebra, let $Q$ be the Stone space of $\mathcal{B}$, let $\delta\colon C(Q)\to L^0(\mathcal{B})$ be a non-zero derivation. Then for every $n\in\mathbb{N}$ there exists $b_n\in C(Q)$, such that $|b_n|\leqslant s(\delta)$ and $|\delta(b_n)|\geqslant ns(\delta)$, in particular, $s(\delta(b_n))=s(\delta)$.
\end{claim}
\begin{proof}
If $0\neq e\leqslant s(\delta),e\in\mathcal{B}$, then there exists $x\in C(Q)$, such that \\$p=es(\delta(x))\neq 0$. Since $y=px\in pC(Q)$, $\delta(y)=p\delta(x)\neq 0$ and \\$\delta(pz)=p\delta(z)\in pL^0(\mathcal{B})$ (see Proposition \ref{der}) for all $z\in C(Q)$, it follows that the restriction $\delta_p$ of the derivation $\delta$ on $pC(Q)$ is a non-zero derivation from $pC(Q)$ into $pL^0(\mathcal{B})$. Let $\Phi$ be the algebraic and lattice isomorphism from $pL^0(\mathcal{B})$ onto  $L^0(\mathcal{B}_p)$, for which $\Phi(g)=g$ for all $g\in p\mathcal{B}$ (see Section \ref{sec_prel}) and let $Q_p$ be the Stone space of the Boolean algebra $\mathcal{B}_p$. It is clear that $\Phi(pC(Q))=C(Q_p)$. Consider the mapping $\hat{\delta}_p\colon C(Q_p)\to L^0(\mathcal{B}_p)$, defined by the equality $\hat{\delta}_p(x)=\Phi(\delta_p(\Phi^{-1}(x))),x\in C(Q_p)$.
Easy to see that $\hat{\delta}_p$ is a linear mapping, in addition, for all  $x,y\in C(Q_p)$ we have
\begin{gather*}
\begin{split}
\hat{\delta}_p(xy)&=\Phi\Bigl(\delta_p\bigl(\Phi^{-1}(x)\Phi^{-1}(y)\bigr)\Bigr)= \Phi\Bigl(\bigl(\delta_p(\Phi^{-1}(x))\Phi^{-1}(y)+\Phi^{-1}(x)\delta_p(\Phi^{-1}(y))\bigr)\Bigr)\\
&=\Phi\Bigl(\delta_p\bigl(\Phi^{-1}(x)\bigr)\Bigr)y +x\Phi\Bigl(\delta_p\bigl(\Phi^{-1}(y)\bigr)\Bigr)=\hat{\delta}_p(x)y+x\hat{\delta}_p(y).
\end{split}
\end{gather*}
Consequently, $\hat{\delta}_p$ is a derivation from $C(Q_p)$ into $L^0(\mathcal{B}_p)$. Since $\delta_p$ is a non-zero derivation, it follows that  $\hat{\delta}_p$ is also a non-zero derivation. By Theorem \ref{sequence}, there exist $a(p,n)\in C(Q_p)$ and $0\neq q_p\in\mathcal{B}_p$ such that $|a(p,n)|\leqslant p$ and $|\hat{\delta}_p(a(p,n))|\geqslant nq_p$. Hence, for $b(p,n)=\Phi^{-1}(a(p,n))$ we have $b(p,n)\in pC(Q),\\ |b(p,n)|\leqslant p$ and
$$|\delta(b(p,n))|=|\delta_p(b(p,n))|=|\Phi^{-1}(\hat{\delta}_p(a(p,n)))|\geqslant nq_p.$$

According to Theorem \ref{principle}, there exist a partition $\{e_i\}_{i\in I}$ of the idempotent $s(\delta)$ and $b(i,n)\in e_iC(Q)$, such that $|b(i,n)|\leqslant e_i$ and $|\delta(b(i,n))|\geqslant ne_i$ for all $i\in I$. For the element $b_n=\underset{i \in I}{\mathrm{mix}}(e_ib(i,n))$ we have that $b_n\in s(\delta)C(Q),\\|b_n|\leqslant s(\delta)$ and $$e_i|\delta(b_n)|=|\delta(e_ib_n)|=|\delta(b(i,n))|\geqslant ne_i$$
for all $i\in I$, and therefore $|\delta(b_n)|\geqslant ns(\delta)$.
\end{proof}

Using Proposition \ref{ad_claim_foralln} the following useful description of the range of values of non-zero derivations, defined on  $C(Q)$, is established.

\begin{theorem}\label{ad_domain} Let $\mathcal{B}$ be a complete Boolean algebra, let $Q$ be the Stone space of $\mathcal{B}$, let $Y$ be a solid space in $L^0(\mathcal{B})$ and let $\delta\colon C(Q)\to Y$ be a non-zero derivation. Then $s(\delta)Y=s(\delta)L^0(\mathcal{B})$, in particular, if $s(\delta)=\mathbf{1}$, then $Y=L^0(\mathcal{B})$.
\end{theorem}
\begin{proof}
By Proposition \ref{der} for every $a\in C(Q)$ we have that \\$\delta\bigl((\mathbf{1}-s(\delta))a\bigl)=\bigl(\mathbf{1}-s(\delta)\bigl)s(\delta(a))\delta(a)=0$, in addition, $\delta(C(Q))\subset s(\delta)Y$. Using Proposition \ref{eY} and passing to the derivation $\delta_{s(\delta)}\colon s(\delta)C(Q)\to s(\delta)Y$, defined by the equality $\delta_{s(\delta)}(a)=\delta(a),a\in s(\delta)C(Q)$, we may assume that $s(\delta)=\mathbf{1}$.

If $Y\neq L^0(\mathcal{B})$, then there exist a partition of unity $\{e_n\}_{n=1}^\infty$ and $x\in L^0(\mathcal{B})$, such that $e_nx=ne_n$ for all $n\in\mathbb{N}$ and $x\notin Y$. Using Proposition \ref{ad_claim_foralln}, select $b_n\in C(Q)$ such that $|b_n|\leqslant\mathbf{1}$ and $|\delta(b_n)|\geqslant n\mathbf{1}$ for all $n\in \mathbb{N}$. Since $|e_nb_n|\leqslant e_n$, we have that $b=\underset{n \in\mathbb{N}}{\mathrm{mix}}(e_nb_n)\in C(Q)$, in addition, $$e_n|\delta(b)|=|\delta(e_nb)|=|\delta(e_nb_n)|=|e_n\delta(b_n)|\geqslant ne_n=e_nx$$ for every $n\in\mathbb{N}$. Thus, $|\delta(b)|\geqslant x\geqslant 0$. Since $\delta(b)\in Y$ and $Y$ is a solid space, it follows that $x\in Y$. This contradiction implies that $Y=L^0(\mathcal{B})$.
\end{proof}

Let us give an example of a derivation $\delta\colon C(Q)\to L^0(\mathcal{B})$ with $s(\delta)=\mathbf{1}$. Let $\mathcal{B}$ be a properly non $\sigma$-distributive Boolean algebra. By Corollary \ref{cor-Kus}, for every non-zero idempotent $e\in\mathcal{B}$ there exists a non-zero derivation \\$\delta_e\colon L^0(\mathcal{B}_e)\to L^0(\mathcal{B}_e)$. Denote by $Q_e$ the Stone space of $\mathcal{B}_e$ and consider the restriction $\hat{\delta}_e$ of the derivation $\delta_e$ on $C(Q_e)$.

We claim that $\hat{\delta}_e$ is a non-zero derivation. Suppose that $\hat{\delta}_e(a)=0$ for all $a\in C(Q_e)$. For every $x\in L^0(\mathcal{B}_e)$ there exists a partition $\{e_n\}_{n=1}^\infty$ of the idempotent $e$, such that $e_nx\in C(Q_e)$ for all $n\in\mathbb{N}$. According to Proposition \ref{der}, we have that
$$e_n\delta_e(x)=\delta_e(e_nx)=\hat{\delta}_e(e_nx)=0, n\in\mathbb{N},$$
that implies the equality $\delta_e(x)=(\sup\limits_{n\in\mathbb{N}} e_n)\delta_e(x)=0, x\in L^0(\mathcal{B}_e)$, which is a contradiction since $\delta_e$ is a non-zero derivation. Consequently, \\
 $\hat{\delta}_e\colon C(Q_e)\to L^0(\mathcal{B}_e)$ is also a non-zero derivation.

Thus, identifying $L^0(\mathcal{B}_e)$ with $eL^0(\mathcal{B})$ and $C(Q_e)$ with $eC(Q)$ (see Section \ref{sec_prel}), we have that for every $0\neq e\in\mathcal{B}$ there exists a non-zero derivation $\delta_e\colon eC(Q)\to eL^0(\mathcal{B})$ with the non-zero support $s(\delta_e)\leqslant e$. For $a\in eC(Q)$ we have $\delta_e\bigl((e-s(\delta_e))a\bigr)=(e-s(\delta_e))s(\delta_e(a))\delta_e(a)=0$, i.e. $\delta_e(a)=\delta_e(s(\delta_e)a)$, and therefore the restriction $\delta_{s(\delta_e)}$ of the derivation $\delta_e$ on $s(\delta_e)C(Q)$ is a non-zero derivation with the support $s(\delta_{s(\delta_e)})=s(\delta_e)$.
By Theorem \ref{principle}, there exist a partition $\{e_i\}_{i\in I}$ of unity $\mathbf{1}$ and non-zero derivations $\delta_{e_i}\colon e_iC(Q)\to e_iL^0(\mathcal{B})$, such that $s(\delta_{e_i})=e_i$ for all $i\in I$.

Now, define the mapping $\delta\colon C(Q)\to L^0(\mathcal{B})$, by setting $\delta(x)=\underset{i\in I}{\mathrm{mix}}(e_i\delta_{e_i}(e_ix)),\\ x\in C(Q)$. It is clear that $\delta$ is a linear mapping, moreover, for all $x,y\in C(Q)$ we have that
\begin{gather*}
\begin{split}
\delta(xy)&=\underset{i\in I}{\mathrm{mix}}(e_i\delta_{e_i}(e_ixe_iy))=
\underset{i\in I}{\mathrm{mix}}(e_i(\delta_{e_i}(e_ix)e_iy+e_ix\delta_{e_i}(e_iy)))\\
&=\underset{i\in I}{\mathrm{mix}}(e_i\delta_{e_i}(e_ix)e_iy)+\underset{i\in I}{\mathrm{mix}}(e_ixe_i\delta_{e_i}(e_iy))=\delta(x)y+x\delta(y).
\end{split}
\end{gather*}
Consequently, $\delta\colon C(Q)\to L^0(\mathcal{B})$ is a derivation, in addition,
$$e_is(\delta(x))=s(e_i\delta(x))=s(\delta_{e_i}(e_ix))$$
for all $x\in C(Q), i\in I$. It means that
\begin{gather*}
\begin{split}
s(\delta)&=\sup\{s(\delta(x)):x\in C(Q)\}=\sup\{e_is(\delta(x)):x\in C(Q),i\in I\}\\
&=\sup\{e_is(\delta_{e_i}(a)):a\in e_iC(Q),i\in I\}=\sup\{e_is(\delta_{e_i}):i\in I\}=\sup\limits_{i\in I} e_i=\mathbf{1}.
\end{split}
\end{gather*}

Thus, we have the following
\begin{claim}\label{prop_sup_iden}
Let $\mathcal{B}$ be a properly non $\sigma$-distributive Boolean algebra, let $Q$ be the Stone space of $\mathcal{B}$. Then there exists a derivation $\delta\colon C(Q)\to L^0(\mathcal{B})$, such that $s(\delta)=\mathbf{1}$.
\end{claim}

Let $X$ be an ideal in $C(Q)$ and let $\delta$ be a derivation from $X$ into $L^0(\mathcal{B})$. Denote by $Y(\delta)$ the smallest solid space in $L^0(\mathcal{B})$, containing the image $\delta(X)$ of the derivation $\delta$.

Proposition \ref{prop_sup_iden} and Theorem \ref{ad_domain} imply that in case when $\mathcal{B}$ is a properly non $\sigma$-distributive Boolean algebra there exists a derivation $\delta\colon C(Q)\to L^0(\mathcal{B})$, such that $Y(\delta)=L^0(\mathcal{B})$.

Let $X$ be an arbitrary ideal in $C(Q)$. A derivation $\delta\colon X\to L^0(\mathcal{B})$ is called \textit{band preserving}, if $s(\delta(x))\leqslant s(x)$ for all $x\in X$ (compare \cite{G-K-K}, 1.1.1). It is clear that for band preserving derivations the inequality $s(\delta)\leqslant s(X)$ always holds. By Proposition \ref{der}, for a derivation $\delta\colon C(Q)\to L^0(\mathcal{B})$ (respectively, $\delta\colon L^0(\mathcal{B})\to L^0(\mathcal{B})$) we have that $\delta(x)=\delta(s(x)x)=s(x)\delta(x)$, i.e.\\ $s(\delta(x))\leqslant s(x)$ for all $x\in C(Q)$ (respectively, $x\in L^0(\mathcal{B})$). Consequently, any derivation from $C(Q)$ (respectively, $L^0(\mathcal{B})$) into $L^0(\mathcal{B})$ is band preserving. At the same time, there exist ideals $X$ in $C(Q)$ with $s(X)\neq\mathbf{1}$ and derivations $\delta\colon X\to L^0(\mathcal{B})$ which are not band preserving (see Example 5.2). Let us give useful criterion for a derivation to be band preserving.

\begin{theorem}\label{band_preser_der}
Let $\mathcal{B}$ be a complete Boolean algebra with the Stone space $Q$, let $X$ be an ideal in  $C(Q)$. For a derivation  $\delta\colon X\to L^0(\mathcal{B})$ the following conditions are equivalent:

(i). $\delta$ is a band preserving derivation;

(ii). $s(\delta)\leqslant s(X)$.
\end{theorem}
\begin{proof}
The implication $(i)\Rightarrow(ii)$ is obvious.

$(ii)\Rightarrow(i)$. Suppose that there exists $x_0\in X$, such that $e=(\mathbf{1}-s(x_0))s(\delta(x_0))\neq 0$. Since $e\leqslant s(\delta(x_0))\leqslant s(\delta)\leqslant s(X)$, there exists $x\in X$, such that $q=es(x)\neq 0$. By Proposition \ref{fund} (i), there exists $0\neq p\in X\cap\mathcal{B}$, such that $p\leqslant q\leqslant e\leqslant s(\delta(x_0))$. Consequently (see Proposition \ref{der}),
$$0\neq ps(\delta(x_0))=s(p\delta(x_0))=s(\delta(px_0))=s(\delta(pex_0))=s(\delta(0))=0.$$
This contradiction implies that $\delta$ is a band preserving derivation.
\end{proof}

\begin{corollary}\label{band_pres_der_FIP}
If $X$ is an ideal in $C(Q)$ with $s(X)=\mathbf{1}$, then any derivation $\delta\colon X\to L^0(\mathcal{B})$ is band preserving.
\end{corollary}

\section{Extension of derivations}\label{sec_ext}
In this section we give the construction of extension of a band preserving derivation $\delta\colon X\to L^0(\mathcal{B})$, acting on an ideal of the algebra $C(Q)$, up to a derivation $\hat{\delta}\colon L^0(\mathcal{B})\to L^0(\mathcal{B})$ (compare \cite{B-Ch-S}, Theorem 3.1).

\begin{theorem}\label{extension}
Let $\mathcal{B}$ be a complete Boolean algebra, let $Q$ be the Stone space of $\mathcal{B}$, let $X$ be an ideal in the algebra $C(Q)$ and let $\delta\colon X\to L^0(\mathcal{B})$ be a band preserving derivation. Then there exists a derivation $\hat{\delta}\colon L^0(\mathcal{B})\to L^0(\mathcal{B})$ such that $\hat{\delta}(x)=\delta(x)$ for all $x\in X$. In addition, if $s(X)=\mathbf{1}$, then such derivation $\hat{\delta}$ is unique.
\end{theorem}
\begin{proof}Firstly, let us assume that $X=C(Q)$. For every $x\in L^0(\mathcal{B})$ there exists a partition of unity $\{e_n\}_{n\in\mathbb{N}}$ such that $e_nx\in C(Q)$ for all $n\in\mathbb{N}$. Set $\hat{\delta}(x)=\underset{n \in \mathbb{N}}{\mathrm{mix}}(e_n\delta(e_nx))$. Let us show that this definition does not depend on a choice of the partition of unity $\{e_n\}_{n\in\mathbb{N}}$. If $\{q_n\}_{n\in\mathbb{N}}$ is another partition of unity, for which  $q_nx\in C(Q)$ for all $n\in\mathbb{N}$ and $y=\underset{n \in \mathbb{N}}{\mathrm{mix}}(q_n\delta(q_nx))$, then
$$e_mq_ny=e_mq_n\delta(q_nx)=q_ne_m\delta(e_mq_nx)=q_ne_m\delta(e_mx)=q_ne_m\hat{\delta}(x)$$
for all $n,m\in\mathbb{N}$. Since $\sup\limits_{m\in\mathbb{N}}e_m=\sup\limits_{n\in\mathbb{N}}q_n=\mathbf{1}$, then $y=\hat{\delta}(x)$. Thus, the mapping $\hat{\delta}\colon L^0(\mathcal{B})\to L^0(\mathcal{B})$ is correctly defined.

If $x,y\in L^0(\mathcal{B})$ and $\{e_n\}_{n\in\mathbb{N}},\{p_n\}_{n\in\mathbb{N}}$ are partitions of unity  such that $e_nx,p_ny\in C(Q)$ for all $n\in\mathbb{N}$, then
\begin{gather*}
\begin{split}
e_np_m\hat{\delta}(x+y)&=e_np_m\delta(e_np_m(x+y))=e_np_m\delta(e_np_mx)+e_np_m\delta(e_np_my)=
\\
&=e_np_m\bigl(\delta(e_nx)+\delta(p_my))=e_np_m(\hat{\delta}(x)+\hat{\delta}(y))
\end{split}
\end{gather*}
for all $n,m\in\mathbb{N}$. Consequently, $\hat{\delta}(x+y)=\hat{\delta}(x)+\hat{\delta}(y)$. Similarly, it is established that  $\hat{\delta}(\lambda x)=\lambda\hat{\delta}(x)$ for all $\lambda\in\mathbb{R}, x\in C(Q)$. Thus, $\hat{\delta}$ is a linear mapping. Further
\begin{gather*}
\begin{split}
e_np_m\hat{\delta}(xy)&=e_np_m\delta((e_nx)(p_my))=e_np_m\bigl(\delta(e_nx)p_my+e_nx\delta(p_my)\bigl)=
\\
&=e_np_m(\hat{\delta}(x)y+x\hat{\delta}(y))
\end{split}
\end{gather*}
for all $n,m\in\mathbb{N}$, that implies  the equality $\hat{\delta}(xy)=\hat{\delta}(x)y+x\hat{\delta}(y)$.

Consequently, $\hat{\delta}\colon L^0(\mathcal{B})\to L^0(\mathcal{B})$ is a derivation, in addition, for $x\in C(Q)$ and a partition of unity  $\{e_n\}_{n\in\mathbb{N}}$ the equalities $e_n\delta(x)=e_n\delta(e_nx)=e_n\hat{\delta}(x),\\n\in\mathbb{N}$ hold, i.e. $\delta(x)=\hat{\delta}(x)$.

Assume that $\delta_1\colon L^0(\mathcal{B})\to L^0(\mathcal{B})$ is another derivation, for which \\$\delta_1(x)=\delta(x)$ for all $x\in C(Q)$. Then, by  Proposition \ref{der}, for $x\in L^0(\mathcal{B})$ and a partition of unity $\{e_n\}_{n\in\mathbb{N}}$ such that $e_nx\in C(Q),n\in\mathbb{N}$ we have $$e_n\delta_1(x)=e_n\delta_1(e_nx)=e_n\hat{\delta}(e_nx)=e_n\hat{\delta}(x)$$ for all $n\in\mathbb{N}$. Since $\sup\limits_{n\in\mathbb{N}}e_n=\mathbf{1}$, it follows that $\delta_1(x)=\hat{\delta}(x)$ for all $x\in L^0(\mathcal{B})$, i.e. $\delta_1=\hat{\delta_1}$.

Now, let $X$ be an arbitrary ideal in $C(Q)$ such that $s(X)=\mathbf{1}$. Due to Proposition \ref{pr6}(ii), $X$ is order-dense solid space in $L^0(\mathcal{B})$. By Proposition \ref{fund}(ii) there exists a partition of unity $\{e_i\}_{i\in I}$ contained in  $X$. For all $i\in I,\lambda\in\mathbb{R}$ we have that $\lambda e_i\in X$, that implies the inclusion $e_iC(Q)\subset X$ for all $i\in I$.

Define the mapping $\overline{\delta}\colon C(Q)\to L^0(\mathcal{B})$ by setting
$$\overline{\delta}(x)=\underset{i \in I}{\mathrm{mix}}(e_i\delta(e_ix)),x\in C(Q).$$
As above it is established that  $\overline{\delta}$ is a derivation from $C(Q)$ into $L^0(\mathcal{B})$, in addition, $\overline{\delta}(x)=\delta(x)$ for all $x\in X$. Assume that $\delta_2\colon C(Q)\to L^0(\mathcal{B})$ is another derivation for which $\delta_2(x)=\delta(x),x\in X$. If $x\in C(Q)$, then $e_ix\in X$ for all $i\in I$ and $$e_i\delta_2(x)=e_i\delta_2(e_ix)=e_i\delta(e_ix)=e_i\overline{\delta}(e_ix)=e_i\overline{\delta}(x),$$
 i.e. $\delta_2=\overline{\delta}$. Thus, $\overline{\delta}$ is a unique derivation from  $C(Q)$ into $L^0(\mathcal{B})$ such that $\overline{\delta}(x)=\delta(x)$ for all $x\in X$.
From the first part of the proof it follows that there exists a unique derivation  $\hat{\delta}\colon L^0(\mathcal{B})\to L^0(\mathcal{B})$ such that $\hat{\delta}(x)=\overline{\delta}(x)=\delta(x)$ for all $x\in X$.

Now, consider an arbitrary ideal $X$ in $C(Q)$. Let $e=s(X)\neq\mathbf{1}$ and $\Phi$ is the algebraic and lattice isomorphism from $eL^0(\mathcal{B})$ onto $L^0(\mathcal{B}_e)$ (see Section \ref{sec_prel}). By Proposition \ref{eY} $\Phi(X)$ is a order-dense solid space in $L^0(\mathcal{B}_e)$. Furthermore, since $s(\delta)\leqslant s(X)$ (see Theorem \ref{band_preser_der}), we have $\delta(X)\subset eL^0(\mathcal{B})=\Phi^{-1}(L^0(\mathcal{B}_e))$.
Consequently, there is correctly defined a mapping $\delta_0\colon \Phi(X)\to L^0(\mathcal{B}_e)$ by the equality $\delta_0(\Phi(x))=\Phi(\delta(x))$. It is clear that $\delta_0$ is a derivation from the order-dense solid space $\Phi(X)$ with values in $L^0(\mathcal{B}_e)$. From the proven above, it follows that there exists a derivation $\hat{\delta}_0\colon L^0(\mathcal{B}_e)\to L^0(\mathcal{B}_e)$ such that  $\hat{\delta}_0(\Phi(x))=\delta_0(\Phi(x))$ for all $x\in X$. Consider the linear mapping $\hat{\delta}_1\colon eL^0(\mathcal{B})\to eL^0(\mathcal{B})$,
 defined by the equality $\hat{\delta}_1(ex)=\Phi^{-1}(\hat{\delta}_0(\Phi(ex))),x\in L^0(\mathcal{B})$. For all  $x,y\in L^0(\mathcal{B})$ we have
\begin{gather*}
\begin{split}
\hat{\delta}_1(exey)&=\Phi^{-1}(\hat{\delta}_0(\Phi(ex)\Phi(ey)))= \Phi^{-1}\bigl(\hat{\delta}_0(\Phi(ex))\Phi(ey)+\Phi(ex)\hat{\delta}_0(\Phi(ey))\bigl)=\\
&=\Phi^{-1}(\hat{\delta}_0(\Phi(ex)))ey+ex\Phi^{-1}(\hat{\delta}_0(\Phi(ey)))=\hat{\delta}_1(ex)ey+ex\hat{\delta}_1(ey).
\end{split}
\end{gather*}

Consequently, $\hat{\delta}_1$ is a derivation from $eL^0(\mathcal{B})$ into $eL^0(\mathcal{B})$, in addition, $$\hat{\delta}_1(x)=\hat{\delta}_1(ex)=\Phi^{-1}(\hat{\delta}_0(\Phi(x)))=\Phi^{-1}(\delta_0(\Phi(x)))=\delta(x)$$ for all $x\in X$.

Extend the derivation $\hat{\delta}_1\colon eL^0(\mathcal{B})\to eL^0(\mathcal{B})$ up to a derivation \\$\hat{\delta}\colon L^0(\mathcal{B})\to L^0(\mathcal{B})$ by setting  $\hat{\delta}(x)=\hat{\delta}_1(ex)$. It is clear that $\hat{\delta}$ is a derivation from $L^0(\mathcal{B})$ into $L^0(\mathcal{B})$ and $\hat{\delta}(x)=\hat{\delta}_1(x)=\delta(x)$ for all $x\in X$.
\end{proof}

Note that Corollary \ref{band_pres_der_FIP} and Theorem \ref{extension} immediately imply the following 

\begin{corollary}\label{extension_fund}
Let $\mathcal{B}$ be a complete Boolean algebra, let $Q$ be the Stone space of $\mathcal{B}$ and let $X$ be a order-dense solid space in $C(Q)$. Then for any \\derivation $\delta\colon X\to L^0(\mathcal{B})$ there exists a unique derivation $\hat{\delta}\colon L^0(\mathcal{B})\to L^0(\mathcal{B})$ such that $\hat{\delta}(x)=\delta(x)$ for all $x\in X$.
\end{corollary}

Let $X$ be an arbitrary nonzero ideal in $C(Q)$ with the support $e=s(X)$.
Let us show that, when the Boolean algebra $\mathcal{B}_e$ is not $\sigma$-distributive, there exists a non-zero band preserving derivation $\delta$ from $X$ into $L^0(\mathcal{B})$. Indeed, in this case, by Corollary \ref{cor-Kus}, there exists a non-zero derivation $\delta\colon L^0(\mathcal{B}_e)\to L^0(\mathcal{B}_e)$. Let $\Phi$ be the algebraic and lattice isomorphism from $eL^0(\mathcal{B})$ onto $L^0(\mathcal{B}_e)$.
Consider the restriction $\delta_0$ of the derivation $\delta$ on $\Phi(X)$. Since $\Phi(X)$ is a order-dense solid space in  $L^0(\mathcal{B}_e)$ (see Proposition \ref{eY}), it follows by corollary \ref{extension_fund} that there exists a unique derivation $\hat{\delta}_0\colon L^0(\mathcal{B}_e)\to L^0(\mathcal{B}_e)$ such that $\hat{\delta}_0(\Phi(x))=\delta_0(\Phi(x))$ for all $x\in X$. Since $\delta(\Phi(x))=\delta_0(\Phi(x))$ for all $x\in X$, due to the uniqueness of the derivation  $\hat{\delta}_0$, we have $\delta=\hat{\delta}_0$. If $\delta_0$ is a trivial derivation, then, due to the construction of $\hat{\delta}_0$ (see the proof of Theorem \ref{extension}), we have that $\hat{\delta}_0$ is also trivial, and therefore $\delta=0$, which is a contradiction. Consequently, $\delta_0$ is a non-zero derivation from $\Phi(X)$ with values in  $L^0(\mathcal{B}_e)$. Construct the mapping $\overline{\delta}\colon X\to L^0(\mathcal{B})$ by setting $\overline{\delta}(x)=\Phi^{-1}(\delta_0(\Phi(x)))$. As in the proof of Theorem \ref{extension}, it is established that the mapping $\overline{\delta}$ is a derivation. In addition, it is clear that $\overline{\delta}$ is a non-zero derivation from $X$ with values in $L^0(\mathcal{B})$ and $s(\delta)\leqslant s(X)$. According to Theorem \ref{band_preser_der}, constructed derivation $\delta$ is a band preserving derivation.

Now, suppose that the Boolean algebra $\mathcal{B}_e$ is $\sigma$-distributive and $\delta$ is a band preserving derivation from the ideal $X$ with values in  $L^0(\mathcal{B})$. Let us show that, in this case, $\delta$ is trivial. Let, as before, $\Phi$ be the algebraic and lattice isomorphism from $eL^0(\mathcal{B})$ onto $L^0(\mathcal{B}_e)$. Since $s(\delta)\leqslant s(X)$ (see Theorem \ref{band_preser_der}), the derivation $\delta_1\colon \Phi(X)\to L^0(\mathcal{B}_e)$ is correctly defined by the equality $\delta_1(\Phi(x))=\Phi(\delta(x)), x\in X$. According to Corollary  \ref{extension_fund}, there exists a derivation $\hat{\delta}_1\colon L^0(\mathcal{B}_e)\to L^0(\mathcal{B}_e)$, such that $\hat{\delta}_1(\Phi(x))=\delta_1(\Phi(x))$ for all $x\in X$. Since the Boolean algebra $\mathcal{B}_e$ is $\sigma$-distributive, by Corollary \ref{cor-Kus}, we have $\hat{\delta}_1=0$, that implies the equality $\delta(x)=0$ for all $x\in X$. Thus, the following version of Theorem \ref{Kus} holds.

\begin{theorem}\label{version-Kus}
Let $\mathcal{B}$ be a complete Boolean algebra, let $Q$ be the Stone space of $\mathcal{B}$, let $X$ be an ideal in  $C(Q)$. The following conditions are equivalent:

(i). The Boolean algebra $\mathcal{B}_{s(X)}$ is $\sigma$-distributive;

(ii). There are no non-zero band preserving derivations from $X$ with values in  $L^0(\mathcal{B})$.
\end{theorem}

Now, consider the complex case. Let $\mathbb{X}$ be an ideal in the algebra $\mathbb{C}(Q)$ and let $(\mathbb{Y},\|\cdot\|_\mathbb{Y})$ be a solid space in  $L^0_\mathbb{C}(\mathcal{B})$. Then $X=\mathbb{X}_h$ is an ideal in the algebra $C(Q)$, $Y=\mathbb{Y}_h$ is a solid space in $L^0(\mathcal{B})$ and $\mathbb{X}=X\varoplus iX, \mathbb{Y}=Y\varoplus iY$. As in case of  $\mathbb{X}=L^0_\mathbb{C}(\mathcal{B})$, a derivation $\delta\colon\mathbb{X}\to\mathbb{Y}$ is called a $*$-derivation, if
$\delta(\overline{x})=\overline{\delta(x)}$ for all $x\in \mathbb{X}$. In this case $\delta(\mathbb{X}_h)\subset \mathbb{Y}_h$ and the restriction of $\delta$ on $X$ is a (real) derivation from $X$ into $Y$.

For any derivation $\delta\colon \mathbb{X}\to \mathbb{Y}$ consider the mappings 
$\delta_\mathbbm{Re}$ and $\delta_\mathbbm{Im}$ from $\mathbb{X}$ into $\mathbb{Y}$, defined by the equalities (\ref{self_adjoint_der}).
Easy to see that $\delta_\mathbbm{Re},\delta_\mathbbm{Im}$ are $*$-derivations from $\mathbb{X}$ into $\mathbb{Y}$, in addition $\delta=\delta_\mathbbm{Re}+i\delta_\mathbbm{Im}$. As in the case of the field $\mathbb{R}$, a derivation $\delta\colon \mathbb{X}\to\mathbb{Y}$ is called  \textit{band preserving}, if $s(\delta(x))\leqslant s(x)$ for all $x\in\mathbb{X}$. It is clear that a derivation $\delta\colon \mathbb{X}\to\mathbb{Y}$ is band preserving if and only if the derivations $\delta_\mathbbm{Re}$ and $\delta_\mathbbm{Im}$ are band preserving, that, in its turn, holds if and only if the restrictions of $\delta_\mathbbm{Re}$ and $\delta_\mathbbm{Im}$ on $X$ are band preserving. Therefore, in particular, for (complex) derivations $\delta\colon \mathbb{X}\to L^0_\mathbb{C}(\mathcal{B})$ the assertions of Theorem \ref{band_preser_der} and Corollary \ref{band_pres_der_FIP} are preserved.

Repeating the proof of Corollary \ref{cor-Kus}, we obtain the following complex variant of Theorem \ref{version-Kus}.
\begin{theorem}\label{version-Kus_complex}
Let $\mathcal{B}$ be a complete Boolean algebra, let $Q$ be the Stone space of  $\mathcal{B}$ 
and let $\mathbb{X}$ be an ideal in $C(Q)$. The following conditions are equivalent:

(i). The Boolean algebra $s(\mathbb{X})\mathcal{B}$ is $\sigma$-distributive;

(ii). Every band preserving derivation from $\mathbb{X}$ into $L^0_\mathbb{C}(\mathcal{B})$ is trivial.
\end{theorem}

Theorems \ref{version-Kus} and \ref{version-Kus_complex} and Corollary \ref{band_pres_der_FIP} imply the following
\begin{corollary}\label{version-Kus-FIP}
For a complete Boolean algebra $\mathcal{B}$ with the Stone space $Q$ and a order-dense solid space $X\subset C(Q)$ (respectively, $\mathbb{X}\subset\mathbb{C}(Q)$) the following conditions are equivalent:

(i). The Boolean algebra $\mathcal{B}$ is $\sigma$-distributive;

(ii). Any derivation from $X$ into $L^0(\mathcal{B})$ (respectively, from
$\mathbb{X}$ into $L^0_\mathbb{C}(\mathcal{B})$) is trivial.
\end{corollary}

Note, that for a solid space $X\subset C(Q)$ with $s(X)\neq\mathbf{1}$ there may exist non-zero derivations $\delta\colon X\to L^0(\mathcal{B})$ even in case when the Boolean algebra $s(X)\mathcal{B}$ is $\sigma$-distributive (see Example 5.2 below). Such derivations $\delta$ are not band preserving, i.e. the inequality  $s(\delta)\leqslant s(X)$ does not hold (see Theorem \ref{version-Kus}).

\section{Main results}\label{sec_main}
In this section the necessary and sufficient conditions for existence of non-zero band preserving derivations from an ideal $X\subset C(Q)$ into a solid space $Y\subset L^0(\mathcal{B})$ are given.

Let $e$ be a non-zero element from a Boolean algebra $\mathcal{B}$. A solid space $Y$ in $L^0(\mathcal{B})$ is said to be \textit{$e$-extended}, if $eY=eL^0(\mathcal{B})$. In this case, according to Proposition \ref{pr6} (i), $eL^0(\mathcal{B})\subset Y$, in particular, $e\in Y$. If $Y$ is a solid space in  $L^0(\mathcal{B}), 0\neq e\in Y\cap\mathcal{B}$ and $\dim(eY)<\infty$, then, by the inclusion $eC(Q)\subset eY$,we have that $\dim C(Q_e)<\infty$, where $Q_e$ is the Stone space of the Boolean algebra $\mathcal{B}_e$. Consequently, the idempotent $e$ is a finite sum of atoms $\{p_i\}_{i=1}^n$, and therefore  $eY=\sum\limits_{i=1}^n p_iY=\sum\limits_{i=1}^n\mathbb{R}p_i=eL^0(\mathcal{B})$, i.e. $Y$ is $e$-extended. Note also, that in case when $\mathcal{B}$ is continuous, a solid space $Y=C(Q)$ is not $e$-extended for any non-zero $e\in\mathcal{B}$.

\begin{theorem}\label{nes_suf_con}
Let $\mathcal{B}$ be a complete Boolean algebra, let $Q$ be the Stone space of  $\mathcal{B}$, let $X$ be an ideal in $C(Q)$ and let $Y$ be a solid space in $L^0(\mathcal{B})$. The following conditions are equivalent:

(i). There exists a non-zero band preserving derivation from $X$ into $Y$;

(ii). There exists a non-zero idempotent $e\leqslant s(X)s(Y)$, such that $Y$ is  $e$-extended and the Boolean algebra $\mathcal{B}_e$ is not $\sigma$-distributive.
\end{theorem}
\begin{proof}
$(i)\Rightarrow(ii)$. Let $\delta\colon X\to Y$ be a non-zero band preserving derivation. By Theorem \ref{band_preser_der} we have that $0\neq s(\delta)\leqslant s(X)s(Y)$. According to Proposition \ref{fund} (i), there exists a non-zero $e\in X\cap\mathcal{B}$, such that $e\leqslant s(\delta)$. Since $eX=eC(Q)$, it follows that the restriction  $\delta_e$ of the derivation  $\delta$ on $eX$ is a derivation from $eC(Q)$ into $eY\subset eL^0(\mathcal{B})$, in addition,
\begin{gather*}
\begin{split}
s(\delta_e)&=\sup\{\delta(ex):x\in X\}=\sup\{e\delta(x):x\in X\}=\\
&=e\sup\{\delta(x):x\in X\}=es(\delta)=e.
\end{split}
\end{gather*}

Let $\Phi$ be the algebraic and lattice isomorphism from $eL^0(\mathcal{B})$ onto $L^0(\mathcal{B}_e),\overline{\delta}(a)=\Phi\bigl(\delta_e(\Phi^{-1}(a))\bigl), a\in \Phi(e C(Q))=C(Q_e)$, where $Q_e$ is the Stone space of the Boolean algebra  $\mathcal{B}_e$. It is clear that $\overline{\delta}$ is a non-zero derivation from $C(Q_e)$ into $L^0(\mathcal{B}_e), s(\overline{\delta})=e$ and $\overline{\delta}(C(Q_e))\subset\Phi(eY)$. Since $\Phi(eY)$ is a solid space in  $L^0(\mathcal{B}_e)$, by Theorem \ref{ad_domain} we have that $\Phi(eY)=L^0(\mathcal{B}_e)$. Hence, $eY=eL^0(\mathcal{B})$, i.e. $Y$ is a $e$-extended solid space in $L^0(\mathcal{B})$.

Moreover, by Theorem \ref{version-Kus}, the Boolean algebra $\mathcal{B}_e$ is not $\sigma$-distributive.

$(ii)\Rightarrow(i)$. According to Theorem \ref{version-Kus}, there exists a non-zero band preserving derivation $\delta_1$ from $\Phi(eX)$ with values in  $L^0(\mathcal{B}_e)=\Phi(eL^0(\mathcal{B}))=\Phi(eY)$, where $\Phi$ is the algebraic and lattice isomorphism from $eL^0(\mathcal{B})$ onto $L^0(\mathcal{B}_e)$. Define the linear mapping $\delta\colon X\to L^0(\mathcal{B})$ by setting $\delta(x)=\Phi^{-1}\bigl(\delta_1(\Phi(ex))\bigl), \\x\in X$. It is clear that $\delta$ is a non-zero derivation from $X$ into \\$\Phi^{-1}(L^0(\mathcal{B}_e))=eY\subset Y$ (the last inclusion follows from Proposition \ref{pr6} (i)). Since $s(\delta)\leqslant e\leqslant s(X)$, we have that $\delta$ is a band preserving derivation (see Theorem \ref{band_preser_der}).
\end{proof}

Now, let us show that any band preserving derivation from an ideal \\$X\subset C(Q)$ with values in a quasi-normed solid space $Y\subset L^0(\mathcal{B})$ is always trivial.

\begin{claim}\label{KIP_e-extended}
If $(Y,\|\cdot\|_Y)$ is a quasi-normed solid space in $L^0(\mathcal{B})$, \\$0\neq e\in\mathcal{B}$ and the Boolean algebra $\mathcal{B}_e$ is not finite, then $Y$ is not an $e$-extended solid space in $L^0(\mathcal{B})$.
\end{claim}
\begin{proof}
Suppose that the quasi-normed solid space $Y$ is an $e$-extended solid space in  $L_0(\mathcal{B})$, i.e. $eY=eL^0(\mathcal{B})$. Since the Boolean algebra $\mathcal{B}_e$ is not finite, there exists a countable partition $\{e_n\}_{n=1}^\infty$ of the idempotent $e$, such that $e_n\neq 0$ for all $n\in\mathbb{N}$. For every $n\in\mathbb{N}$ select $m_n\in\mathbb{N}$ such that $m_n\|e_n\|_Y>n$. After that, select $y\in eL^0(\mathcal{B})=eY\subset Y$, such that $e_ny=m_ne_n$. We have
$$ n<m_n\|e_n\|_Y=\|m_ne_n\|_Y=\|e_ny\|_Y\leqslant\|y\|_Y$$
for all $n\in\mathbb{N}$, that is impossible.

Consequently, $Y$ is not an $e$-extended solid space in $L^0(\mathcal{B})$.
\end{proof}

\begin{theorem} \label{main}Let $\mathcal{B}$ be a complete Boolean algebra, let $Q$ be the Stone space of $\mathcal{B}$, let $X$ be an ideal in $C(Q)$ and let $Y$ be a quasi-normed solid space in  $L^0(\mathcal{B})$. Then any band preserving derivation $\delta\colon  X\to Y$ is trivial.
\end{theorem}
\begin{proof}
Suppose that there exists a non-zero band preserving derivation $\delta\colon X\to Y$. According to Theorem \ref{nes_suf_con}, there exists a non-zero idempotent $e\leqslant s(X)s(Y)$, such that $Y$ is an $e$-extended solid space in  $L^0(\mathcal{B})$ and the Boolean algebra $\mathcal{B}_e$ is not $\sigma$-distributive. Hence, the Boolean algebra $\mathcal{B}_e$ is not finite and, by Proposition \ref{KIP_e-extended}, $Y$ is not an $e$-extended solid space in $L^0(\mathcal{B})$. From the obtained contradiction it follows that any band preserving derivation $\delta\colon X\to Y$ is trivial.
\end{proof}

Let us give one useful corollary of Theorem  \ref{nes_suf_con}.
\begin{corollary}\label{der_ideals}
Let $\mathcal{B}$ be an arbitrary complete Boolean algebra, let $Q$ be the Stone space of $\mathcal{B}$ and let $X,Y$ be ideals in $C(Q)$. Then any band preserving derivation $\delta$ from $X$ into $Y$ is trivial.
\end{corollary}
\begin{proof}Suppose that there exists a non-zero band preserving derivation $\delta\colon X\to Y$. By Theorem \ref{nes_suf_con} there exists a non-zero idempotent $e\leqslant s(X)s(Y)$, such that $eL^0(\mathcal{B})=eY\subset eC(Q)\subset eL^0(\mathcal{B})$ and the Boolean algebra $e\mathcal{B}$ is not  $\sigma$-distributive. Consequently, $eC(Q)=eL^0(\mathcal{B})$, that implies the finiteness of the Boolean algebra $e\mathcal{B}$, in particular, the Boolean algebra $e\mathcal{B}$ is $\sigma$-distributive, that is not true.
\end{proof}

\textbf{Example 5.1.}
Firstly, let us give an example of order-dense solid space $Y$ in $L^0(\mathcal{B})$, which does not coincide with $L^0(\mathcal{B})$, such that there exist a non-zero band preserving derivation from $C(Q)$ into $Y$.

Let $\mathcal{B}$ be a properly non $\sigma$-distributive Boolean algebra, let $Q$ be the Stone space of $\mathcal{B}$. Select an infinite partition  $\{e_i\}_{i\in I}$ of unity of the Boolean algebra $\mathcal{B}$ with $e_i\neq 0, i\in I$,  and denote by $Y$ the set of all $y\in L^0(\mathcal{B})$, such that $e_iy=0$ for all $i\in I$, possibly, except a finite subset $I(y)\subset I$.

It is clear that $Y$ is a order-dense solid space in $L^0(\mathcal{B})$, in addition, $\mathbf{1}
\notin Y$, i.e. $Y\neq L^0(\mathcal{B})$. Moreover, $e_iY=e_iL^0(\mathcal{B})$, i.e. $Y$ is an $e_i$-extended order-dense solid space for all $i\in I$. Since the Boolean algebra $\mathcal{B}$ is properly non $\sigma$-distributive, it follows that the Boolean algebra $e_i\mathcal{B}$ is not $\sigma$-distributive for all $i\in I$. From Theorem \ref{nes_suf_con} and the equalities $s(Y)=\mathbf{1}=s(C(Q))$, it follows that there exists a non-zero band preserving derivation from $C(Q)$ into $Y$. $\Box$

\textbf{Example 5.2.}
Now, give an example of a quasi-normed solid space \\$Y\subset C(Q)$ and an ideal $X\subset C(Q)$ with $s(X)\neq\mathbf{1}$, such that there exists a non-zero derivation $\delta\colon X\to Y$ with the property $s(\delta)s(X)=0$.

Let $\mathcal{B}=2^\Delta$ be a complete atomic Boolean algebra with the countable set of atoms $\Delta=\{q_n\}_{n=1}^\infty$, let $Q$ be the Stone of  $\mathcal{B}$.  In this case, the algebra $C(Q)$ is identified with the algebra $l_\infty$ of all bounded sequences of real numbers and the algebra $L^0(\mathcal{B})$ coincides with the algebra $l_0$ of all sequences from  $\mathbb{R}$.

Select an element $x_0=\{\alpha_n\}_{n=1}^\infty\in l_\infty$, such that $\alpha_n=\frac{1}{k}$ for $n=2k$ and $\alpha_k=0$ for $n=2k-1, k\in\mathbb{N}$. Consider the principal ideal $X=x_0l_\infty$ in $l_\infty$, such that $s(X)=e=\sup\limits_{k\geq 1} q_{2k}$. If $e\in X$, then $e=x_0y$ for some  $y=\{y_n\}\in l_\infty$, in particular, $1=\frac{1}{k}y_{2k}, k\in\mathbb{N}$, i.e. $y_{2k}=k\rightarrow\infty$, that contradicts to the inclusion  $y\in l_\infty$. Consequently, $e\notin X$. 

It is clear that $X^2=\{\sum\limits_{i=1}^n a_ib_i:a_i,b_i\in X,i=\overline{1,n}, n\in\mathbb{N}\}$ is an ideal in $l_\infty$ and $X^2\subset X$. If $x_0\in X^2$, then $x_0=x_0^2\sum\limits_{i=1}^nu_iv_i$ for some $u_i,v_i\in l_\infty, i=\overline{1,n},\\ n\in\mathbb{N}$, and therefore $e=x_0\sum\limits_{i=1}^n u_iv_i\in X$, that is not true. Hence, $\mathbb{R}x_0\cap X^2=\{0\}$. Using the Zorn lemma, construct the Hamel basis $\mathcal{E}$ in  $X$, containing $x_0$ and the Hamel basis in $X^2$. Then $X=\mathbb{R}x_0\varoplus E$, where  $E$ is a linear subspace in $X$, generated by the set $\mathcal{E}\setminus\{x_0\}$, in particular, $X^2\subset E$.

Consider a solid space $Y=(\mathbf{1}-e)l_\infty$ in $l_\infty$ with the monotone Banach norm $\|y\|_Y=\|y\|_\infty, y\in Y$. Define a linear mapping  $\delta\colon X\to Y$ by setting $\delta(\alpha x_0+z)=\alpha(\mathbf{1}-e)$ for all $\alpha x_0+z\in\mathbb{R}x_0\varoplus E=X, \alpha\in\mathbb{R},z\in E$. If $a=\alpha x_0+y, b=\beta x_0+z\in X, \alpha,\beta\in\mathbb{R}, y,z\in E$, then $ab\in X^2\subset E,\\ y,z\in X=eX$, and therefore $\delta(ab)=0, \delta(a)b=\alpha (\mathbf{1}-e)eb=0,\\ a\delta(b)=ae\beta(\mathbf{1}-e)=0$. Consequently, $\delta(ab)=\delta(a)b+a\delta(b)$, i.e. $\delta$ is a non-zero derivation from $X$ into the Banach solid space $Y\subset l_\infty$, for which $s(\delta)s(X)=0$. By Theorem \ref{band_preser_der} the constructed derivation  $\delta\colon X\to Y$ is not band preserving.

Now, consider the complex case. Let $\mathbb{X}$ be an ideal in the algebra $\mathbb{C}(Q)$ and let $(\mathbb{Y},\|\cdot\|_\mathbb{Y})$ be a solid space in  $L^0_\mathbb{C}(\mathcal{B})$. Since any derivation  $\delta\colon\mathbb{X}\to\mathbb{Y}$ may be represented as  $\delta=\delta_\mathbbm{Re}+i\delta_\mathbbm{Im}$, where $\delta_\mathbbm{Re},\delta_\mathbbm{Im}$ are $*$-derivations from the solid space $\mathbb{X}$ into $\mathbb{Y}$, Theorems \ref{nes_suf_con}, \ref{main} and Corollary \ref{der_ideals} imply the following

\begin{corollary}\label{main-cor}Let $\mathcal{B}$ be an arbitrary complete Boolean algebra, let $Q$ be the Stone space of  $\mathcal{B}$, let $\mathbb{X}$ be an ideal in $\mathbb{C}(Q)$, and let $\mathbb{Y}$ be a solid space in $L^0_\mathbb{C}(\mathcal{B})$. Then

(i). There exists a non-zero band preserving derivation $\delta\colon\mathbb{X}\to\mathbb{Y}$ if and only if there exists a non-zero idempotent $e\leqslant s(\mathbb{X})s(\mathbb{Y})$, such that $e\mathbb{Y}=eL^0_\mathbb{C}(\mathcal{B})$ and the Boolean algebra  $\mathcal{B}_e$ is not $\sigma$-distributive;

(ii). If $(\mathbb{Y},\|\cdot\|_Y)$ is a quasi-normed solid space, or $\mathbb{Y}\subset\mathbb{C}(Q)$, then any band preserving derivation  $\delta\colon \mathbb{X}\to \mathbb{Y}$ is trivial.
\end{corollary}

Now, give a version of Corollary  \ref{main-cor} for commutative  $AW^*$-algebras. An $AW^*$-algebra is a $C^*$-algebra which is simultaneously a Baer $*$-algebra (\cite{book-Kusraev}, 7.5.2). If $\mathcal{A}$ is a commutative  $AW^*$-algebra, then the lattice \\$P(\mathcal{A})=\{p\in\mathcal{A}:p=p^*=p^2\}$ of all projections from $\mathcal{A}$ is a complete Boolean algebra \cite{Chilin}, in particular, if $\mathcal{A}$ is a commutative von Neumann algebra, then the Boolean algebra $P(\mathcal{A})$ is multinormed. A commutative $AW^*$-algebra $\mathcal{A}$ is $*$-isomorphic to the $C^*$-algebra $\mathbb{C}(Q(P(\mathcal{A})))$, where $Q(P(\mathcal{A}))$ is the Stone space of the Boolean algebra $P(\mathcal{A})$ (\cite{book-Kusraev}, 7.4.3, 7.5.2). Denote by $S(\mathcal{A})$ the $*$-algebra of all measurable operators affiliated with $AW^*$-algebra $\mathcal{A}$ (see e.g. \cite{Berberian}, \cite{Chilin}). It is known \cite{Berberian}, that for a commutative $AW^*$-algebra $\mathcal{A}$ the  $*$-algebra $S(\mathcal{A})$ is $*$-isomorphic to the $*$-algebra $L^0_\mathbb{C}(P(\mathcal{A}))$. Therefore, Theorem \ref{version-Kus}, Corollaries \ref{cor_atomic-distrib} and \ref{main-cor} imply the following

\begin{theorem} Let $\mathcal{A}$ be a commutative $AW^*$-algebra (respectively,  commutative von Neumann algebra), let $\mathcal{I}$ be an ideal in  $\mathcal{A}$, and let $\mathbb{Y}$ be a solid space $S(\mathcal{A})$. Then

(i). The Boolean algebra $P(s(\mathcal{I})\mathcal{A})$ of all projections from $s(\mathcal{I})\mathcal{A}$ is $\sigma$-distributive (respectively, atomic) if and only if there no non-zero band preserving derivations from $\mathcal{I}$ into $S(\mathcal{A})$;

(ii). There exist a non-zero band preserving derivation  $\delta\colon\mathcal{I}\to\mathbb{Y}$ if and only if there exist a non-zero projection $p\leqslant s(\mathcal{I})s(\mathbb{Y})$, such that $p\mathbb{Y}=pS(\mathcal{A})$ and the Boolean algebra $P(p\mathcal{A})$ is not $\sigma$-distributive;

(iii). If $(\mathbb{Y},\|\cdot\|_\mathbb{Y})$ is a quasi-normed solid space in $S(\mathcal{A})$, or $\mathbb{Y}\subset\mathcal{A}$, then any band preserving derivation $\delta$ from $\mathcal{I}$ into $\mathbb{Y}$ is trivial.
\end{theorem}

Give one more illustration of Theorem \ref{main}. Let $(\Omega,\Sigma,\mu)$ be a measurable space with a complete  $\sigma$-finite measure $\mu$, let $L^0(\Omega)$ be the algebra of all measurable real-valued functions defined on $(\Omega,\Sigma,\mu)$, and let $L^\infty(\Omega)$ be the subalgebra of all essentially bounded functions from  $L^0(\Omega)$ (functions that are equal almost everywhere are identified). Denote by $t_\mu$ the topology of convergence locally in measure $\mu$ in $L^0(\Omega)$. Convergence $f_n\xrightarrow{t_\mu} f,f_n,f\in L^0(\Omega)$ means that $f_n\chi_A\rightarrow f\chi_A$ in measure $\mu$ for any $A\in\Sigma$ with $\mu(A)<\infty$. Proposition \ref{der} and density of the subalgebra of step functions in the algebra $L^0(\Omega)$ with respect to the topology $t_\mu$ imply that any  $t_\mu$-continuous derivation \\$\delta\colon L^0(\Omega)\to L^0(\Omega)$ is trivial.

Consider an arbitrary non-zero ideal $X$ in the algebra $L^\infty(\Omega)$ and normed solid space $(Y,\|\cdot\|_Y)$ of measurable functions on $(\Omega,\Sigma,\mu)$ (see e.g. \cite{K-A}, ch.\setcounter{gl}{4}\Roman{gl}, \S 3). The examples of such solid spaces are $L_p$-spaces, $p\geqslant 1$, the Orlicz, Marcinkiewicz spaces, symmetric spaces of measurable functions on  $(\Omega,\Sigma,\mu)$ \cite{K-P-S}. Theorems \ref{version-Kus} and \ref{main} imply the following

\begin{theorem}[compare \cite{B-S_d}]\label{main_measurable_space}
(i). Any band preserving derivation $\delta$ from $X$ into $Y$ is trivial, in particular, if $X$ is a order-dense solid space, then any derivation  $\delta\colon X\to Y$ is trivial;

(ii). If $(\Omega,\Sigma,\mu)$ is continuous measure space, then there exist a non-zero band preserving derivation from $X$ into $L^0(\Omega)$, which is not continuous with respect to the topology $t_\mu$;

(iii). If $(\Omega,\Sigma,\mu)$  is an atomic measure space, then any band preserving derivation from $\delta$ from $X$ into $L^0(\Omega)$ is trivial.
\end{theorem}

Note again, the condition that a derivation $\delta$ is band preserving is necessary for validity of Theorem \ref{main_measurable_space} (i), (iii). Recall that in Example 5.2 it is constructed a non-zero derivation $\delta$ from an ideal $X\subset L^\infty(\Omega)$ with values in a Banach solid space $Y\subset L^\infty(\Omega)$ for an atomic measure space $(\Omega,\Sigma,\mu)$, such that $s(\delta)s(X)=0$ (the last equality means that the derivation $\delta$  is not band preserving (see Theorem \ref{band_preser_der})).


\end{document}